\documentclass[10pt]{amsart}
\usepackage{latexsym, amscd, amsfonts, eucal, mathrsfs, amsmath, mathabx, amssymb, amsthm, xypic,xr, stmaryrd, color, enumerate, tikz, todonotes}
\usepackage{appendix}
\usepackage{multicol}
\usepackage[all]{xy}
\usepackage{hyperref}
\usepackage{tikz-cd}
\usepackage{times}

\newtheorem*{thm*}{Theorem}

\newtheorem{theorem}{Theorem}[section]
\newtheorem{lemma}[theorem]{Lemma}
\newtheorem{proposition}[theorem]{Proposition}
\newtheorem{corollary}[theorem]{Corollary}

\theoremstyle{definition}

\newtheorem{definition}[theorem]{Definition}
\newtheorem{construction}[theorem]{Construction}

\newtheorem{warning}[theorem]{Warning}
\newtheorem{remark}[theorem]{Remark}

\usepackage{cleveref}
\usepackage{scrextend}

\begin{document}

\title{Examples of Disk Algebras}
\author{Sanath Devalapurkar, Jeremy Hahn, Tyler Lawson,
Andrew Senger, and Dylan Wilson}
\date{}
\maketitle

\begin{abstract} We produce refinements of the known
multiplicative structures on the Brown-Peterson spectrum $\mathrm{BP}$,
its truncated variants $\mathrm{BP}\langle n\rangle$, Ravenel's spectra $X(n)$, and evenly graded polynomial rings over the sphere spectrum. Consequently, topological Hochschild homology relative to these rings inherits
a circle action.
\end{abstract}

\tableofcontents

\section{Introduction}

If $R$ is a ring spectrum, then the algebraic $K$-theory of $R$ is often understood by means of its trace maps to $\mathrm{TC}^{-}(R)=\mathrm{THH}(R)^{hS^1}$, $\mathrm{TP}(R)=\mathrm{THH}(R)^{tS^1}$, and $\mathrm{TC}(R)$.  To compute any of these invariants, it has proven extremely fruitful to approximate the absolute Hochschild homology 
$\mathrm{THH}(R)$ by Hochschild homology relative to some other base, i.e. perform descent along a map
\[
    \mathrm{THH}(R) \to \mathrm{THH}(R/A).    
\]
For example, this is one of the main ideas behind
the definition of prismatic cohomology of ring
spectra given in \cite{even-filtration}, and is featured in the foundational \cite[\S 11]{bms2}.  Works such as \cite{k-theory-of-Zpn, liu-wang, krause-nikolaus-dvr, lee-thh, hahn-wilson} showcase both the computational and theoretical effectiveness of the technique.

To enact the above strategy, one needs $A$ to admit enough structure that $\mathrm{THH}(R/A)$ exists as an $S^1$-equivariant spectrum.
The action of $S^1$ on $\mathbb{R}^2$ by rotation defines an $S^1$-action on the operad $\mathbb{E}_2$, and hence an $S^1$-action on
the category $\mathsf{Alg}_{\mathbb{E}_2}$ of
$\mathbb{E}_2$-algebras. The 
structure necessary on $A$ to define an $S^1$-action on $\mathrm{THH}(R/A)$ is that of a homotopy
fixed point for this action.\footnote{This appears to be
folklore, but we sketch a short proof due to Asaf Horev in
Corollary \ref{cor:s1-action}.} The category
$\mathsf{Alg}_{\mathbb{E}_2}^{hS^1}$ goes by many
names, such as the category
of framed $\mathbb{E}_2$-algebras,
$\mathbb{E}_2\rtimes S^1$-algebras,
or $\mathbb{E}_{\mathrm{BU}(1)}$-algebras. 
Following \cite{ayala-francis}, we will call these 
\textit{$\mathsf{Disk}_2^{\mathrm{BU}(1)}$-algebras}; more generally, there is a notion of $\mathsf{Disk}_n^B$-algebra
which we review below.

Here, we prove that several familiar and fundamental ring spectra admit extra structure of
this form:
\begin{theorem}[Corollary \ref{cor:xn-disk}, Corollary \ref{cor:cyclotomic-base}, Corollary \ref{cor:bp-disk4}, Theorem \ref{thm:bpn-disk3}]
We have:
\begin{enumerate}
    \item At any prime $p$, $\mathrm{BP}$ admits the structure of a $\mathsf{Disk}_4^{\mathrm{BU}(2)}$-$\mathrm{MU}$-algebra.
     \item At any prime $p$ and for each integer $n \ge 0$, there is a form of $\mathrm{BP}\langle n\rangle$ which is a $\mathsf{Disk}_3^{\mathrm{BU}(1)}$-$\mathrm{MU}$-algebra.
	\item For each integer $n \ge 0$, the Ravenel spectrum $X(n)$ admits the structure of a $\mathsf{Disk}_2^{\mathrm{BU}(1)}$-algebra.
    \item For any integer $n$, the spherical polynomial algebra $\mathbb{S}[x_{2n}]$ on a degree $2n$ class admits the structure of a $\mathsf{Disk}_2^{\mathrm{BU}(1)}$-algebra.
\end{enumerate}
\end{theorem}

\begin{remark}
There has been a long history of work equipping the above ring spectra with highly structured multiplications. For example, Basterra--Mandell \cite{basterra-mandell} proved that $\mathrm{BP}$ admits a unique $\mathbb{E}_{4}$-algebra structure, and in \cite{hahn-wilson} the second and fifth authors show that there are $\mathbb{E}_3$-$\mathrm{MU}$-algebra forms of $\mathrm{BP}\langle n \rangle$.  The above theorem strengthens these results, and can be seen as part of the general effort to equip ring spectra with the maximum possible amount of structure.
\end{remark}

\subsection*{Acknowledgements}
The authors would like to thank 
Andrew Blumberg,
Michael Hill,
Michael Hopkins,
Achim Krause,
Michael Mandell,
and
Thomas Nikolaus
for helpful conversations. We thank Ferdinand Wagner for pointing out that $\mathrm{THH}(A/B)$ does not generally carry a $B$-module structure when $B$ is $\mathsf{Disk}_2$, and we thank Asaf Horev and Alexander Kupers for discussions about Remark 2.2 and Corollary 2.8.
During the course of this work, Sanath Devalapurkar was supported by NSF grant DGE-2140743, Jeremy Hahn was supported by NSF grant DMS-1803273, Andrew Senger was supported by NSF grant DMS-2103236, Tyler Lawson was supported by NSF grant DMS-2208062, and Dylan Wilson was supported by NSF grant DMS-1902669.
This project was made possible by the hospitality of AIM.


\section{Review of Disk Algebras}

\subsection{Definitions} We recall the algebraic setup
from \cite{ayala-francis}.

\begin{definition} Let $B \in \mathsf{Spaces}_{/\mathrm{BTop}(n)}$.
Then $\mathsf{Disk}_n^B$ \cite[Definition 2.9]{ayala-francis}
is the symmetric monoidal ($\infty$-)category of 
$n$-manifolds homeomorphic to finite disjoint unions
of $n$-dimensional Euclidean spaces equipped with
a lift of the classifer of their tangent microbundle to $B$.
The symmetric monoidal structure is given by disjoint union.
The category of $\mathsf{Disk}_n^B$-algebras in a 
symmetric monoidal category $\mathcal{C}$ is defined
as
	\[
	\mathsf{Alg}_{\mathsf{Disk}_n^B}(\mathcal{C}):=
	\mathsf{Fun}^{\otimes}(\mathsf{Disk}_n^B, \mathcal{C}).
	\]
\end{definition}

\begin{remark}\label{rmk: disk algebras as a limit} The symmetric monoidal category
$\mathsf{Disk}_n^B$ is the symmetric monoidal
envelope of the $\infty$-operad $\mathbb{E}_B$
of \cite[5.4.2.10]{ha}. Combining 
\cite[2.2.4.9, 2.3.3.4]{ha}, we learn that the map $B \to \mathrm{BTop}(n)$
produces a local system of categories of $\mathbb{E}_n$-algebras
and that there is an equivalence:
	\[
	\mathsf{Alg}_{\mathsf{Disk}_n^B}(\mathcal{C})
	\simeq \lim_B \mathsf{Alg}_{\mathbb{E}_n}(\mathcal{C}).
	\]
Note that this depends on \cite[Remark 2.3.3.4]{ha}, which is stated without proof; recently, detailed proofs appeared as \cite[Proposition 1.2]{arakawa} and \cite[Theorem 2.2]{KrannichKupers}.
\end{remark}

\subsection{Disk algebras in spaces} Given any
pointed space $X$, the functor of compactly-supported
maps
	\[
	\mathrm{Map}_c(-, X):
	\mathsf{Disk}_n \to 
	\mathsf{Spaces}
	\]
is symmetric monoidal for the structure of disjoint union
on the source and cartesian product on the target.
Observe that, upon restriction to the full subcategory spanned
by $\mathbb{R}^n$, we obtain the local system
	\[
	\mathrm{BTop}(n) \to \mathsf{Spaces}
	\]
associated to the action of $\mathrm{Top}(n)$
on $\mathrm{Map}_c(\mathbb{R}^n, X) = \Omega^nX$.
We will denote this local system by $\Omega^{\lambda_n}X$.

\begin{proposition} The above construction refines
to an adjunction
	\[
	\xymatrix{
	\mathsf{Alg}_{\mathsf{Disk}_n^B}(\mathsf{Spaces})
	\ar@<.5ex>[r]^-{\mathrm{B}^{\lambda_n}} &
	\mathsf{Psh}(B)_* \ar@<.5ex>[l]^-{\Omega^{\lambda_n}}.
	}
	\]
This restricts to an equivalence between
group-like algebras and pointwise $n$-connective presheaves.
\end{proposition}
\begin{proof} In the discussion above we produced a functor
	\[
	\mathsf{Spaces}_* \to \mathsf{Alg}_{\mathsf{Disk}_n^B}(
	\mathsf{Spaces}) = \lim_B \mathsf{Alg}_{\mathbb{E}_n}(
	\mathsf{Spaces}).
	\]
given by $X \mapsto \Omega^{\lambda_n}X$. 
This is the same data as a map
	\[
	\mathsf{Spaces}_* \to \mathsf{Alg}_{\mathbb{E}_n}(\mathsf{Spaces})
	\]
of presheaves on
$B$, where the source is regarded as a constant presheaf.
Taking global sections then produces the desired functor
$\Omega^{\lambda_n}$. The existence of a left adjoint
and the restricted equivalence is a formal consequence
of the known statement applied pointwise on $B$.
\end{proof}

We will also need the following computation.

\begin{lemma} Suppose $X= \Omega^{\infty}M$
is an infinite-loop space given the structure of
a $\mathsf{Disk}_{n}^B$-algebra by restriction.
Then $\mathrm{B}^{\lambda_n}X = \Omega^{\infty}\Sigma^{\lambda_n}M$.
\end{lemma}
\begin{proof} As in the previous proposition,
observe that the construction $Y \mapsto \Omega^{\lambda_n}Y$ refines to a functor
(which we temporarily give alternative notation)
    \[
    \Pi^{\lambda_n}: \mathsf{Fun}(B,\mathsf{Sp})
    \to \mathsf{Alg}_{\mathsf{Disk}_n^B}(\mathsf{Sp}^{\times})
    \]
that intertwines $\Omega^{\infty}$. Here we have
decorated $\mathsf{Sp}$ with $\times$ to indicate
that we are using the cartesian monoidal structure.
Since $\mathsf{Sp}$ is stable, this coincides
with the cocartesian monoidal structure and thus by
\cite[2.4.3.8]{ha}
the forgetful functor
    \[
    \mathsf{Alg}_{\mathsf{Disk}_n^B}(\mathsf{Sp}^{\times})
    \to 
    \mathsf{Fun}(B,\mathsf{Sp})
    \]
is an equivalence. By design, the composite
of $\Pi^{\lambda_n}$ with this forgetful functor
is $\Omega^{\lambda_n}$. In other words: the two potentially different (additive) $\mathsf{Disk}_n^B$-algebra structures on the spectrum
$\Omega^{\lambda_n}Y$ must coincide for any
local system of spectra on $B$. 

It then follows that
$\Omega^{\lambda_n}\Omega^{\infty}\Sigma^{\lambda_n}M
\simeq X$ as $\mathsf{Disk}_n^B$-algebras, which proves the result.
\end{proof}

\begin{warning} If $X$ is an $\mathbb{E}_{\infty}$-space
then, when regarded as a $\mathsf{Disk}_n^B$-algebra,
the underlying presheaf on $B$ is constant.
However, the presheaf $\mathrm{B}^{\lambda_{n}}X$
need not be constant. For example, if $X = \mathbb{Z}$
and $B = \mathrm{BO}(1)$, then $\mathrm{B}^{\lambda_1}\mathbb{Z}
= S^{\lambda_1}$ is the one-point compactification of the
sign representation. This is not (Borel) equivariantly
trivial, as seen, for example, from its integral homology.
\end{warning}

\subsection{Factorization homology}

Recall from \cite{ayala-francis} that $\mathsf{Mfld}_n^B$
denotes the symmetric monoidal
($\infty$-)category of $B$-framed manifolds,
which contains $\mathsf{Disk}_n^B$ as a full subcategory.

\begin{definition} Let $A$ be a $\mathsf{Disk}_n^B$-algebra in a presentably
symmetric monoidal $\infty$-category $\mathcal{C}$. We define
the factorization homology functor
	\[
	\int_{(-)}A : \mathsf{Mfld}^B_n \to \mathcal{C}
	\]
by left Kan extension along the inclusion
$\mathsf{Disk}_n^B \hookrightarrow \mathsf{Mfld}_n^B$.
\end{definition}

This functor gives a generalization of Hochschild homology
by the following theorem.

\begin{theorem}[Ayala-Francis, Lurie]
If $A$ is a $\mathsf{Disk}_1^B$-algebra in $\mathcal{C}$, then
there is a canonical, $S^1$-equivariant equivalence
	\[
	\int_{S^1}A \simeq \mathrm{HH}(A),
	\]
where the latter is defined via the cyclic bar construction in
$\mathcal{C}$. 
\end{theorem}

Encoding factorization homology as a functor on
$B$-framed manifolds now allows us to equip relative
Hochschild homology with a circle action.

The following argument is due to Asaf Horev in private communication.

\begin{corollary}[Horev]\label{cor:s1-map}
    If $A$ is a $\mathsf{Disk}_{n+2}^{\mathrm{BU}(1)}$-algebra in $\mathcal{C}$ (so that $A$ admits an $S^1 = \mathrm{U}(1)$-action), then there is
a natural $S^1$-equivariant $\mathbb{E}_n$-$\mathrm{HH}(A)$-algebra structure on $A$.
\end{corollary}

\begin{proof}
    By Dunn additivity, we have $\mathsf{Alg}_{\mathsf{Disk}_{n+2}^{\mathrm{BU}(1)}} (\mathcal{C}) \simeq \mathsf{Alg}_{\mathsf{Disk}_{2}^{\mathrm{BU}(1)}} (\mathsf{Alg}_{\mathbb{E}_n}(\mathcal{C})).$
    Replacing $\mathcal{C}$ by $\mathsf{Alg}_{\mathbb{E}_n}(\mathcal{C}),$ we may assume that $n=0$.
    
    Consider the annulus $\mathrm{Ann} \cong S^1 \times (0,1)$, and the disk $D^2$ as objects of $\mathsf{Mfld}_2 ^{BU(1)},$ the category of oriented surfaces and embeddings.
    Note that $\int_{\mathrm{Ann}} A = \int_{S^1 \times (0,1)} A \simeq \mathrm{HH}(A)$
    and $\int_{D^2} A \simeq A$.
    
    The category $\mathsf{Mfld}_2 ^{BU(1)}$ is a symmetric monoidal category under disjoint union.
    Moreover, taking an embedding $\coprod_k (0,1) \hookrightarrow (0,1)$ to its product with $S^1$ to obtain an embedding $\coprod_k S^1 \times (0,1) \hookrightarrow S^1 \times (0,1)$ equips $\mathrm{Ann}$ with an $S^1$-equivariant $\mathbb{E}_1$-algebra structure, where the $S^1$-equivariance comes from rotation.
    Similarly, taking an embedding $\coprod_k (0,1) \coprod (0,1] \hookrightarrow (0,1]$ to the corresponding radial embedding $\coprod_k S^1 \times (0,1) \coprod D^2 \hookrightarrow D^2$
    equips $D^2$ with the structure of an $S^1$-equivariant left module over $\mathrm{Ann}$.
    Finally, the inclusion $\mathrm{Ann} \hookrightarrow D^2$ is a map of $S^1$-equivariant left modules over $\mathrm{Ann}$.

    Applying the symmetric monoidal functor
    \[
	\int_{(-)}A : \mathsf{Mfld}^{\mathrm{BU(1)}}_2 \to \mathcal{C},
	\]
    we obtain a map of $S^1$-equivariant $\mathrm{HH}(A)$-modules $\mathrm{HH}(A) \to A$, i.e. the structure of an $S^1$-equivariant $\mathbb{E}_0$-$\mathrm{HH}(A)$-algebra structure on $A$.
\end{proof}

\begin{remark}
Concretely, the unit map
\[\mathrm{HH}(A) \to A\]
of the above $\mathbb{E}_n$-algebra structure may be obtained by functoriality of factorization homology with respect to the inclusion
$(\mathbb{R}^2-\{0\})\times \mathbb{R}^n \to
\mathbb{R}^{2} \times \mathbb{R}^n$.  Note that the map $\mathbb{R}^2-\{0\} \to \mathbb{R}^2$ can be identified with the inclusion $\mathbb{C}-\{0\} \to \mathbb{C}$ of $\mathrm{U}(1)$-equivariant complex manifolds and hence $\mathrm{U}(1)$-equivariant $\mathrm{BU}(1)$-framed manifolds, where $\mathbb{C}$ is viewed as $\mathrm{U}(1)$-equivariant via multiplication by complex units.
\end{remark}


\begin{corollary}\label{cor:s1-action} If $A$ is a $\mathsf{Disk}_{2}^{\mathrm{BU}(1)}$-algebra and $\mathcal{D}$ is an $A$-linear category, then
$\mathrm{THH}(\mathcal{D}/A)$ has a canonical $S^1$-action.
\end{corollary}
\begin{proof} By the previous corollary, base change
along $\mathrm{THH}(A) \to A$ takes $S^1$-equivariant
$\mathrm{THH}(A)$-modules to $S^1$-equivariant spectra.
Whence the claim for
	\[
	\mathrm{THH}(\mathcal{D}/A) =
	\mathrm{THH}(\mathcal{D})\otimes_{\mathrm{THH}(A)}A.\qedhere
	\]
\end{proof}

\begin{remark}
The non-$S^1$-equivariant analogue of 
Corollary \ref{cor:s1-map}
is proved as \cite[Lemma 4.6]{krause-nikolaus-lectures}.
The statement of Corollary \ref{cor:s1-map} for $n=2$ was also mentioned in \cite[Remark 3.4]{yuan-redshift}, where it was attributed to Asaf Horev.
\end{remark}

\begin{warning}
    In \Cref{cor:s1-action}, it is \textit{not} true that $\mathrm{THH}(\mathcal{D}/A)$ admits the structure of an $A$-module, even nonequivariantly. The basic issue can already be seen ``one level down'': if $A$ is a $\mathsf{Disk}_{1}^{\mathrm{BO}(1)}$-algebra (that is, an associative algebra equipped with an anti-involution) and $B$ is a unital left $A$-module, one can form the tensor product $B \otimes_A B^\mathrm{op}$. (This can be viewed as ``$\mathbb{E}_0$-Hochschild homology'', i.e., the factorization homology over the $0$-sphere of $B$ over $A$.) This tensor product admits a natural $\mathbb{Z}/2$-action, but does not admit an $A$-module structure (even nonequivariantly).
\end{warning}

\section{Thom spectra}

In this section we explain how to equip Thom spectra
with $\mathsf{Disk}$-algebra structures in certain
situations.

\subsection{Disk algebras in spaces and orientability}\label{subsec: disk algebras in spaces}

We now observe that orientability allows us to automatically upgrade
some $\mathbb{E}_n$-algebras to $\mathsf{Disk}$-algebras.

\begin{proposition}\label{prop:disk-and-orient}
Suppose $X = \Omega^{\infty}E$ and that there is
a chosen equivalence
$\Sigma^{\lambda_n}E \simeq \Sigma^nE$ of local systems
on $B$. Then every group-like
$\mathbb{E}_n$-algebra $Y$ equipped with a $\mathbb{E}_n$-algebra map $Y \to X$ has a canonical
refinement to a $\mathsf{Disk}_{n}^B$-algebra over $X$.
\end{proposition}
\begin{proof} By assumption, there is a map of spaces
	\[
	\mathrm{B}^{n}Y \to \mathrm{B}^{n}X.
	\]
Using our assumption that $\Sigma^{\lambda_n}E \simeq \Sigma^{n}E$, we obtain equivalences:
	\[
	\mathrm{B}^{n}X
	\simeq
	\Omega^{\infty}\Sigma^{n}E
	\simeq \Omega^{\infty}\Sigma^{\lambda_n}E
	\simeq \mathrm{B}^{\lambda_{n}}X.
	\]
Thus we get a map of $\mathsf{Disk}_{2n}^{\mathrm{BU}(n)}$-algebras
	\[
	\Omega^{\lambda_{n}}\mathrm{B}^{n}Y \to X
	\]
which refines the original map.
\end{proof}

\begin{warning} The action of $\Omega B$ on $Y$ constructed
above may be nontrivial.
\end{warning}

\begin{corollary}\label{cor:alg-over-bu} Let $Y$ be a group-like $\mathbb{E}_{2n}$-algebra
over $\mathrm{BU} \times \mathbb{Z}$. Then $Y$ has a canonical
refinement to
a $\mathsf{Disk}_{2n}^{\mathrm{BU}(n)}$-algebra over $\mathrm{BU} \times \mathbb{Z}$.
\end{corollary}
\begin{proof} 
As $\mathrm{ku}$ is a module
over $\mathrm{MU}$, a choice of Thom class
for the bundle $\lambda_{2n}$ over $\mathrm{BU}(n)$ gives the
desired $\mathrm{U}(n)$-equivariant
equivalence 
	\[
	\Sigma^{\lambda_{2n}}\mathrm{ku}
	= \Sigma^{\lambda_{2n}}\mathrm{MU}\otimes_{\mathrm{MU}}
	\mathrm{ku}
	\simeq
	\Sigma^{2n}\mathrm{MU}\otimes_{\mathrm{MU}}
	\mathrm{ku} \simeq
	\Sigma^{2n}\mathrm{ku}.\qedhere
	\]
\end{proof}

The same argument using the Atiyah-Bott-Shapiro orientation
gives the following.

\begin{corollary}
Let $Y$ be a group-like $\mathbb{E}_{n}$-algebra
over $\mathrm{BO} \times \mathbb{Z}$. Then $Y$ has a canonical
refinement to
a $\mathsf{Disk}_{n}^{\mathrm{BSpin}(n)}$-algebra over $\mathrm{BO} \times \mathbb{Z}$.
\end{corollary}

\subsection{Main Result} Let $\mathcal{C}$
be presentably symmetric monoidal and denote by
$\mathrm{Pic}(\mathcal{C})$ the $\mathbb{E}_{\infty}$-groupoid
of $\otimes$-invertible objects in $\mathcal{C}$.
Recall that, for any
map $X \to \mathrm{Pic}(\mathcal{C})$,
there is a unique extension to a colimit-preserving functor
	\[
	\mathrm{Th}_{\mathcal{C}}:
	\mathsf{Spaces}_{/X} \to \mathcal{C};
	\]
if the map $X \to \mathrm{Pic}(\mathcal{C})$ is one of $\mathbb{E}_\infty$-spaces, this functor is lax symmetric monoidal.
See, e.g. \cite[Proposition 3.1.3]{hopkins-lurie-brauer} and \cite[Section 7.1]{categorified-traces}.
We will be mainly concerned with the cases
$\mathcal{C} = \mathsf{Sp}$ and $\mathcal{C} = \mathsf{Sp}_{(p)}$.
For any $\infty$-operad $\mathcal{O}$, we then get
an induced functor
	\[
	\mathrm{Th}_{\mathcal{C}}:
	\mathsf{Alg}_{\mathcal{O}}(\mathsf{Spaces}_{/X})
	\simeq
	\mathsf{Alg}_{\mathcal{O}}(\mathsf{Spaces})_{/X}
	\to
	\mathsf{Alg}_{\mathcal{O}}(\mathcal{C}).
	\]

Applying the results from the previous subsection,
we immediately deduce the following.

\begin{theorem}\label{thm:thom-presentable} Let $\mathcal{C}$ be
presentably symmetric monoidal.
Suppose $X=\Omega^{\infty}E$, 
there is a chosen trivialization $\Sigma^{\lambda_n}E \simeq
\Sigma^nE$ of local systems on $B$, and
and we are given an $\mathbb{E}_{\infty}$-algebra
map $X \to \mathrm{Pic}(\mathcal{C})$.
Suppose $\xi: Y \to X$ is a map of group-like
$\mathbb{E}_n$-algebras. Then $\mathrm{Th}_{\mathcal{C}}(\xi)$
admits a canonical $\mathsf{Disk}_n^B$-algebra structure.
\end{theorem}

\begin{corollary}\label{cor:bu-thom-disk} Suppose $\xi: Y \to \mathrm{BU} \times \mathbb{Z}$
is a map of group-like $\mathbb{E}_{2n}$-algebras.
Then the Thom spectrum
$\mathrm{Th}(\xi)$ admits a canonical
$\mathsf{Disk}_{2n}^{\mathrm{BU}(n)}$-algebra structure.
\end{corollary}

\begin{corollary}\label{cor:xn-disk} Let $X(n)$ be the Ravenel spectrum from \cite[Section 3]{ravenel-loc}.
Then $X(n)$ admits the structure of an $\mathsf{Disk}_2^{\mathrm{BU}(1)}$-algebra.
\end{corollary}
\begin{proof} By definition, $X(n)$ is the Thom spectrum of the double loop map
$\Omega^2\mathrm{BSU}(n) \to \Omega^2\mathrm{BSU}
\simeq\mathrm{BU}$. 
\end{proof}

\begin{remark}
Recall that there is a truncated form of the Quillen idempotent 
$\epsilon_m$ on $X(p^m)_{(p)}$ (see 
\cite[Proposition 1.3.7]{hopkins-thesis}). We will write $T(m)$ 
to denote the resulting summand of $X(p^m)_{(p)}$, so that 
$T(m)$ approximates $\mathrm{BP}$
in the same way as $X(m)$ approximates $\mathrm{MU}$.
At $p=2$, $T(1)$
 admits the structure of an 
 $\mathsf{Disk}_2^{\mathrm{BU}(1)}$-algebra. Indeed, in this 
 case, $T(1) = X(2)$, so the result follows from Corollary
 \ref{cor:xn-disk}. At 
 $p=2$, it is also known that $T(2)$ admits the structure of an 
 $\mathbb{E}_2$-ring. Using Corollary
 \ref{cor:bu-thom-disk}, one can show that $T(2)$ in fact 
 admits the structure of an $\mathsf{Disk}_2^{\mathrm{BU}(1)}$-algebra: indeed, by 
 \cite[Remark 3.1.9]{bpn-thom}, it is the Thom spectrum of the double 
 loop map $\mu:\Omega \mathrm{Sp}(2)\to \mathrm{BU}$
 obtained from taking double 
 loops of the composite
	\[
	\mathrm{BSp}(2)\to \mathrm{BSU}(4)\to \mathrm{BSU} \simeq 
	\mathrm{B}^3\mathrm{U}.
	\]
\end{remark}

Corollary \ref{cor:bu-thom-disk} can also be used to study polynomial rings over the sphere spectrum.
Recall the following construction, e.g., from \cite[Construction 4.1.1]{hahn-wilson} (see also \cite[Section 3.4]{rotinv}).
\begin{construction}
Fix an integer $n\in \mathbb{Z}$, and let $\mathbb{Z}^\mathrm{ds}$ denote the constant simplicial set associated to the set of integers. Then, the free graded $\mathbb{E}_1$-ring $\mathbb{S}[x_{2n, 1}]$ on a class in degree $2n$ and weight $1$ admits the structure of a graded $\mathbb{E}_2$-ring. This can be viewed as an $\mathbb{E}_2$-monoidal functor $\iota_n: \mathbb{Z}^\mathrm{ds} \to \mathsf{Sp}$ sending $1\mapsto S^{2n}$; this functor factors through the inclusion $\mathsf{Pic(Sp)} \to \mathsf{Sp}$. Let us write $\mathbb{S}[x_{2n}]$ to denote the underlying $\mathbb{E}_2$-ring in $\mathsf{Sp}$.

Using the $\mathbb{E}_2$-monoidal functor $\iota_n$, one can define a spectral analogue of the ``shearing'' functor on graded spectra. The following is an adaptation of \cite[Proposition 3.3.4]{arpon-thesis}. Let $\mathcal{C}$ be a stable presentably symmetric monoidal $\infty$-category, and let $\mathcal{C}^\mathrm{gr} = \mathsf{Fun}(\mathbb{Z}^\mathrm{ds}, \mathcal{C})$ denote the $\infty$-category of graded objects in $\mathcal{C}$. The composite
\begin{equation}\label{eq: E2-monoidal shearing}
    \mathbb{Z}^\mathrm{ds} \times \mathcal{C}^\mathrm{gr} \xrightarrow{\iota_n \times \mathrm{ev}} \mathsf{Pic(Sp)} \times \mathcal{C} \xrightarrow{\otimes} \mathcal{C}
\end{equation}
is a lax $\mathbb{E}_2$-monoidal functor. Using the universal property of Day convolution, this in turn defines a lax $\mathbb{E}_2$-monoidal functor $\mathrm{sh}_n: \mathcal{C}^\mathrm{gr} \to \mathcal{C}^\mathrm{gr}$ which acts on a graded object by $M(\bullet) \mapsto M(\bullet)[2n\bullet]$. It is easy to see that this functor is in fact $\mathbb{E}_2$-monoidal and defines an equivalence $\mathrm{sh}_n: \mathcal{C}^\mathrm{gr} \xrightarrow{\sim} \mathcal{C}^\mathrm{gr}$. In fact, $\mathrm{sh}_n \simeq \mathrm{sh}_1^{\circ n}$.
\end{construction}
\begin{proposition}\label{prop:shearing-disk2-monoidal}
The shearing equivalence $\mathrm{sh}_n: \mathcal{C}^\mathrm{gr} \xrightarrow{\sim} \mathcal{C}^\mathrm{gr}$ admits the structure of a $\mathsf{Disk}_2^{\mathrm{BU}(1)}$-monoidal functor.
\end{proposition}
\begin{proof}
It suffices to show that the composite \cref{eq: E2-monoidal shearing} admits the structure of a $\mathsf{Disk}_2^{\mathrm{BU}(1)}$-monoidal functor. The map $\otimes: \mathsf{Pic(Sp)} \times \mathcal{C}^\mathrm{gr} \to \mathcal{C}$ is evidently symmetric monoidal, so it in turn suffices to show that $\iota_n$ admits the structure of a $\mathsf{Disk}_2^{\mathrm{BU}(1)}$-monoidal functor. But $\iota_n$ can be factored as the composite
$$\mathbb{Z}^\mathrm{ds} \xrightarrow{\cdot n} \mathbb{Z}^\mathrm{ds} \to \mathrm{BU} \times \mathbb{Z}^\mathrm{ds} \to \mathsf{Pic(Sp)},$$
where the second map is the inclusion of the factor in the product. The inclusion $\mathbb{Z}^\mathrm{ds} \to \mathrm{BU} \times \mathbb{Z}^\mathrm{ds}$ is one of group-like $\mathbb{E}_2$-algebras (for instance, it can be obtained via Bott periodicity by taking double loops of the map $\mathrm{BU}(1) \to \mathrm{BU}$), so the claim follows from the discussion in \cref{subsec: disk algebras in spaces}.
\end{proof}
\begin{remark}
The functor $\mathrm{sh}_1$ does \textit{not} admit an $\mathbb{E}_3$-monoidal structure. Otherwise, $\mathbb{S}[x_2]$ would admit the structure of an $\mathbb{E}_3$-algebra in $\mathsf{Sp}$. To see that this is impossible, observe that if $\mathbb{S}[x_2]$ did admit the structure of an $\mathbb{E}_3$-algebra, the class $x_2^2: S^4 \to \mathbb{S}[x_2]$ would factor as
$$\xymatrix{
& S^4 \ar[d] \ar[dr] & \\
\Sigma^2 \mathbb{R}P_2^4 \ar[r]^-\sim & \mathrm{Conf}_2(\mathbb{R}^3)_+ \otimes_{\Sigma_2} (S^2)^{\otimes 2} \ar@{-->}[r] & \mathbb{S}[x_2].
}$$
Composing with the projection $\mathbb{S}[x_2] \to S^4$, this would show that the bottom cell of $\Sigma^2 \mathbb{R}P_2^4$ is unattached; but this is false, since the $4$- and $6$-cells of $\Sigma^2 \mathbb{R}P_2^4$ are connected by $\eta$.

It is easier to show that the inclusion $\mathbb{Z}^\mathrm{ds} \to \mathrm{BU} \times \mathbb{Z}^\mathrm{ds}$ is not a map of $\mathbb{E}_3$-algebras. Otherwise, taking the $3$-fold bar construction would show that there is a map $K(\mathbb{Z}, 3) \to \mathrm{SU}$ which is an isomorphism on $\mathrm{H}^3(-; \mathbb{Z})$. This is impossible: for instance, the resulting composite
$$K(\mathbb{Z}, 3) \to \mathrm{SU} \to K(\mathbb{Z}, 3)$$
would be nonzero on $\mathrm{H}^6(-; \mathbb{Z})$; but $\mathrm{H}^6(\mathrm{SU}; \mathbb{Z}) = 0$.
\end{remark}
\begin{corollary}\label{cor:cyclotomic-base}
Let $j\in \mathbb{Z}$. Then, $\mathbb{S}[x_{2j, 1}]$ admits the structure of a $\mathsf{Disk}_2^{\mathrm{BU}(1)}$-algebra in $\mathsf{Sp}^\mathrm{gr}$.
\end{corollary}
\begin{proof}
Let $\mathbb{S}[x_{0,1}]=\Sigma^{\infty}_+\mathbb{N}$ denote the free $\mathbb{E}_1$-algebra in graded spectra on a class in weight $1$ and degree zero; this in fact admits the structure of an $\mathbb{E}_\infty$-ring in $\mathsf{Sp}^\mathrm{gr}$, and $\mathbb{S}[x_{2j, 1}] \simeq \mathrm{sh}_j \mathbb{S}[x_{0,1}]$. Since $\mathrm{sh}_j$ is $\mathsf{Disk}_2^{\mathrm{BU}(1)}$-monoidal by Proposition \ref{prop:shearing-disk2-monoidal}, this implies that $\mathbb{S}[x_{2j, 1}]$ admits the structure of a $\mathsf{Disk}_2^{\mathrm{BU}(1)}$-algebra in $\mathsf{Sp}^\mathrm{gr}$. \qedhere

\end{proof}
\begin{remark}
There is an evident generalization of Corollary \ref{cor:cyclotomic-base} to multi-graded $\mathsf{Disk}_2^{\mathrm{BU}(1)}$-algebras in several variables.
\end{remark}

\section{Retracts of Complex Bordism}

The following result allows us to equip $\mathrm{BP}$
with a $\mathsf{Disk}$-algebra structure.

\begin{theorem} Every $\mathbb{E}_4$-algebra map
$\mathrm{MU}_{(p)} \to \mathrm{MU}_{(p)}$ 
refines to a $\mathsf{Disk}_4^{\mathrm{BU}(2)}$-algebra map.
\end{theorem}
\begin{proof} We would like to show that the map
	\[
	\mathrm{Map}_{\mathsf{Disk}_2^{\mathrm{BU}(2)}}(\mathrm{MU}_{(p)},
	\mathrm{MU}_{(p)}) \to
	\mathrm{Map}_{\mathbb{E}_4}(\mathrm{MU}_{(p)},
	\mathrm{MU}_{(p)})
	\]
is surjective on path components. Under the identification
$\mathsf{Alg}_{\mathsf{Disk}_2^{\mathrm{BU}(2)}} \simeq
\mathsf{Alg}_{\mathbb{E}_2}^{h\mathrm{U}(2)}$,
we may identify this map with the inclusion of fixed points
	\[
	\mathrm{Map}_{\mathbb{E}_4}(\mathrm{MU}_{(p)},
	\mathrm{MU}_{(p)})^{h\mathrm{U}(2)}
	\to
	\mathrm{Map}_{\mathbb{E}_4}(\mathrm{MU}_{(p)},
	\mathrm{MU}_{(p)})
	\]
for some $\mathrm{U}(2)$-action on the source.
To show that this is surjective on path components,
it will suffice to prove that the homotopy fixed
point spectral sequence for the source collapses.
For this, it further suffices to prove that
$\pi_*\mathrm{Map}_{\mathbb{E}_4}(\mathrm{MU}_{(p)},
	\mathrm{MU}_{(p)})$
is concentrated in even degrees.
Using the Thom isomorphism of \cite[Corollary 3.18]{omar-toby}, we have:
	\begin{align*}
	\mathrm{Map}_{\mathbb{E}_4}(\mathrm{MU}_{(p)},
	\mathrm{MU}_{(p)}) &\simeq
	\mathrm{Map}_{\mathbb{E}_4}(\mathrm{MU},
	\mathrm{MU}_{(p)}) \\
	&\simeq
	\mathrm{Map}_{\mathbb{E}_4}(\Sigma^{\infty}_+\mathrm{BU},
	\mathrm{MU}_{(p)}) \\
	&\simeq 
	\mathrm{Map}_*(\mathrm{BU}\langle 6\rangle,
	\mathrm{B}^4\mathrm{GL}_1\mathrm{MU}_{(p)})\\
	&\simeq
	\mathrm{Map}(\Sigma^{-4}
	\Sigma^{\infty}_+\mathrm{BU}\langle 6\rangle,
	\mathrm{gl}_1\mathrm{MU}_{(p)}).
	\end{align*}
Since $\mathrm{BU}\langle 6\rangle$ has an even
cell decomposition, and the homotopy of
$\mathrm{gl}_1\mathrm{MU}_{(p)}$ is concentrated in
even degrees, the Atiyah-Hirzebruch spectral
sequence collapses and the answer is concentrated
in even degrees. This completes the proof.
\end{proof}

\begin{corollary}\label{cor:bp-disk4}
$\mathrm{BP}$ admits the structure
of a $\mathsf{Disk}_4^{\mathrm{BU}(2)}$-algebra under
$\mathrm{MU}$.
\end{corollary}
\begin{proof} Apply the previous theorem to
the $\mathbb{E}_4$-algebra idempotent produced
by Basterra-Mandell in \cite{basterra-mandell}.
\end{proof}

\begin{warning} Unlike the $\mathsf{Disk}$-algebra structures
produced on Thom spectra, the refinements of the self-maps
of $\mathrm{MU}$, and hence of the algebra structure on
$\mathrm{BP}$, are highly non-canonical.
\end{warning}

\section{Truncated Brown-Peterson Spectra}

In \cite{hahn-wilson}, the second and fifth authors
produced $\mathbb{E}_3$-$\mathrm{MU}$-algebra forms of
$\mathrm{BP}\langle n\rangle$. In this section,
we explain how to modify the argument in loc. cit.
to produce
$\mathsf{Disk}_3^{\mathrm{BU}(1)}$-$\mathrm{MU}$-algebra
structures.

\subsection{Review of obstruction theory} If $\mathcal{O}$
is an operad, then the deformation theory of
an $\mathcal{O}$-algebra $A$ is governed by the cotangent
complex, which is an operadic module over $A$.
In the case of interest, it follows from \cite[7.3.4.13]{ha}
that the cotangent complex of a $\mathsf{Disk}_n^B$-algebra
$A$ lies in
	\[
	\mathsf{Mod}_A^{\mathsf{Disk}_n^B}(\mathcal{C}):=
	\lim_B \mathsf{Mod}_A^{\mathbb{E}_n}(\mathcal{C}).
	\]
The category of $\mathbb{E}_n$-$A$-modules
is equivalent to the category of modules over the enveloping
algebra
	\[
	\mathcal{U}^{(n)}(A):= \int_{\mathbb{R}^n-\{0\}}A,
	\]
so an alternative perspective on $\mathsf{Disk}_n^B$-$A$-modules
is via the equivalence
	\[
	\mathsf{Mod}_A^{\mathsf{Disk}_n^B}(\mathcal{C})
	\simeq
	\mathsf{Mod}_{\mathcal{U}^{(n)}(A)}(\mathsf{Fun}(B, \mathcal{C})).
	\]

In these terms, the cotangent complex and enveloping
algebra are related to one another using the following
theorem:

\begin{theorem}[Lurie, Francis] \label{thm:en-cotangent} If $A$ is a
$\mathsf{Disk}_n^B$-algebra, then
there is a fiber sequence
	\[
	\mathcal{U}^{(n)}(A)
	\to A
	\to \Sigma^{\lambda_n}\mathbb{L}_A
	\]
of $\mathsf{Disk}_n^B$-$A$-modules.
\end{theorem}
\begin{proof} The proof given in \cite[Theorem 2.26]{francis}
applies verbatim for general $B$.
\end{proof}

\subsection{Main result}

\begin{theorem}\label{thm:bpn-disk3}
There are forms of $\mathrm{BP}\langle n\rangle$
which are 
$\mathsf{Disk}_3^{\mathrm{BU}(1)}$-$\mathrm{MU}$-algebras,
and for which the maps
	\[
	\mathrm{BP}\langle n\rangle \to
	\mathrm{BP}\langle n-1\rangle
	\]
are maps of $\mathsf{Disk}_3^{\mathrm{BU}(1)}$-$\mathrm{MU}$-algebras.
\end{theorem}
\begin{proof} The proof in
\cite[Theorem 2.0.6]{hahn-wilson} goes through
\textit{mutatis mutandis} using the description
of the cotangent complex above, except that
we replace the use of \cite[Theorem 2.5.5]{hahn-wilson}
with the following refinement: for any virtual complex
representation $V$ of $\mathrm{U}(1)$, the
spectrum
	\[
	\left(\Sigma^V\mathrm{map}_{\mathcal{U}^{(3)}_{\mathrm{MU}}
	(\mathrm{BP}\langle n\rangle)}(\mathrm{BP}\langle n\rangle,
	\mathrm{BP}\langle n\rangle)\right)^{h\mathrm{U}(1)}
	\]
has homotopy groups concentrated in even degrees;
moreover, the map
	\[
	\pi_*\left(\Sigma^V\mathrm{map}_{\mathcal{U}^{(3)}_{\mathrm{MU}}
	(\mathrm{BP}\langle n\rangle)}(\mathrm{BP}\langle n\rangle,
	\mathrm{BP}\langle n\rangle)\right)^{h\mathrm{U}(1)}
	\to
	\pi_*
	\mathrm{map}_{\mathcal{U}^{(3)}_{\mathrm{MU}}
	(\mathrm{BP}\langle n\rangle)}(\mathrm{BP}\langle n\rangle,
	\mathrm{BP}\langle n\rangle)
	\]
is surjective.

In fact, this statement is an immediate consequence
of Theorem 2.5.5. in \textit{loc. cit.}: since the homotopy
groups of $\mathrm{map}_{\mathcal{U}^{(3)}_{\mathrm{MU}}
	(\mathrm{BP}\langle n\rangle)}(\mathrm{BP}\langle n\rangle,
	\mathrm{BP}\langle n\rangle)$
(and hence of any even suspension)
are concentrated in even degrees, the homotopy fixed
point spectral sequence for $\mathrm{U}(1)$ collapses
and is again concentrated in even degrees.
\end{proof}

\bibliographystyle{amsalpha}
\bibliography{Bibliography}

\end{document}